\documentclass[12pt]{amsart}
\usepackage{latexsym}
\usepackage{amssymb}
\usepackage{amsthm}
\theoremstyle{plain}
\usepackage[colorlinks]{hyperref}

\newtheorem{theorem}{Theorem}[section]
\newtheorem*{theorem*}{Theorem}

\newtheorem{lemma}[theorem]{Lemma}
\newtheorem{proposition}[theorem]{Proposition}
\newtheorem{corollary}[theorem]{Corollary}

\newtheorem*{thompson}{Thompson's~Theorem}
\newtheorem*{navarro-tiep}{Navarro-Tiep's~Theorem}

\theoremstyle{definition}

\theoremstyle{remark}

\numberwithin{equation}{theorem}

\newcommand{\PGL}{{\mathrm {PGL}}}
\newcommand{\SL}{{\mathrm {SL}}}
\newcommand{\PSL}{{\mathrm {PSL}}}

\newcommand{\PSU}{{\mathrm {PSU}}}

\newcommand{\Ker}{\operatorname{Ker}}
\newcommand{\Aut}{{\mathrm {Aut}}}

\newcommand{\Irr}{{\mathrm {Irr}}}

\newcommand{\Stab}{{\mathrm {Stab}}}

\newcommand{\RR}{{\mathbb R}}
\newcommand{\QQ}{{\mathbb Q}}

\newcommand{\FF}{{\mathbb F}}

\newcommand{\acd}{\mathrm{acd}}

\newcommand{\odd}{\mathrm{odd}}

\newcommand{\bC}{{\mathbf{C}}}
\newcommand{\bO}{{\mathbf{O}}}
\newcommand{\bN}{{\mathbf{N}}}
\newcommand{\bZ}{{\mathbf{Z}}}
\newcommand{\Al}{\textup{\textsf{A}}}
\newcommand{\Sy}{\textup{\textsf{S}}}

\begin{document}

\title[Characters of $p'$-degree and Thompson's character degree
theorem] {Characters of $p'$-degree\\ and Thompson's character
degree theorem}

\thanks{The author gratefully acknowledges the support of the NSA Young
  Investigator Grant \#H98230-14-1-0293 and a Faculty Scholarship
  Award from Buchtel College of Arts and Sciences, The University
  of Akron. He is also grateful to Alexander Moret\'{o} for several helpful discussions}

\author{Nguyen Ngoc Hung}
\address{Department of Mathematics, The University of Akron, Akron,
Ohio 44325, USA} \email{hungnguyen@uakron.edu}



\date{June 21, 2015}

\begin{abstract}
A classical theorem of John Thompson on character degrees asserts
that if the degree of every ordinary irreducible character of a
finite group $G$ is 1 or divisible by a prime $p$, then $G$ has a
normal $p$-complement. We obtain a significant improvement of this
result by considering the average of $p'$-degrees of irreducible
characters. We also consider fields of character values and prove
several improvements of earlier related results.
\end{abstract}

\maketitle


\section{Introduction}

One of the classical results on character degrees is the celebrated
theorem of J.\,G.~Thompson, which asserts that if the degree of
every ordinary irreducible character of a finite group $G$ is 1 or
divisible by a prime $p$, then $G$ has a normal $p$-complement, see
\cite{Thompson} or \cite[Corollary~12.2]{Isaacs1}. Let
$\acd_{p'}(G)$ denote the average of $p'$-degrees of irreducible
characters of $G$. Then this result can be reformulated as follows:

\begin{thompson} Let $G$ be a finite group. If $\acd_{p'}(G)=1$ then $G$ has a normal
$p$-complement.\end{thompson}

In this paper, we significantly improve Thompson's theorem in the
point of view of $\acd_{p'}$ and investigate further the relation
between characters of $p'$-degree and $p$-nilpotency.

\begin{theorem}\label{theorem-main-1}
Let $p$ be an odd prime and $G$ a finite group. We have
\begin{enumerate}
\item[(i)] if $\acd_{2'}(G)<3/2$ then $G$ has a normal
$2$-complement, and
\item[(ii)] if $\acd_{p'}(G)<4/3$ then $G$ has a normal
$p$-complement.
\end{enumerate}
\end{theorem}

\noindent We emphasize that, in contrast to Thompson's theorem where
it is required that $G$ has no nontrivial character degrees coprime
to $p$ at all, in Theorem~\ref{theorem-main-1} we allow $G$ to have
nontrivial character degrees coprime to $p$, and we still can
conclude the $p$-nilpotency of $G$ as long as the number of linear
characters of $G$ is large enough.

A deep part in the proof of Theorem~\ref{theorem-main-1} is to prove
the solvability of the groups in consideration. In fact, we obtain
the following.

\begin{theorem}\label{theorem-main-2}
Let $p>5$ be a prime and $G$ a finite group. If one of the following
happens
\begin{enumerate}
\item[(i)] $\acd_{2'}(G)<3$,
\item[(ii)] $\acd_{3'}(G)<3$,
\item[(iii)] $\acd_{5'}(G)<11/4$,
\item[(iv)] $\acd_{p'}(G)<16/5$,
\end{enumerate} then $G$ is solvable.
\end{theorem}

Given a finite group $G$, one can always find a prime $5<p\nmid |G|$
so that $\acd_{p'}(G)$ is simply the average degree of \emph{all}
irreducible characters of $G$. Therefore
Theorem~\ref{theorem-main-2}(iv) refines and improves on the main
result of \cite{Moreto-Nguyen}. We remark that, as illustrated by
the nonsolvable groups $\Al_5$ and $\SL(2,5)$, all the bounds in
Theorem~\ref{theorem-main-2} are best possible. Though the bounds in
Theorem~\ref{theorem-main-1} also cannot be improved when $p=2$ or
$3$, as shown by $\Al_4$ and $\Sy_3$, we think that the correct
bound when $p>2$ is $(2p+2)/(p+3)$, attained at the dihedral group
of order $2p$.

It has been shown in recent works that there is a close connection
between important characteristics of finite groups such as
nilpotency, supersolvability, solvability, or $p$-solvability and
several invariants concerning character degrees such as the average
character degree, the character degree sum, the largest character
degree, or the character degree ratio, see
\cite{Cossey-Nguyen,Cossey-Halasi-Maroti-Nguyen,Magaard-Tongviet,{Isaacs-Loukaki-Moreto},Lewis-Nguyen,Moreto-Nguyen,Maroti-Nguyen,Qian,Halasi-Hannusch-Nguyen}.
Theorems~\ref{theorem-main-1} and \ref{theorem-main-2} reinforce
this phenomenon for characters of $p'$-degree.

Several refinements of Thompson's theorem have been proposed in the
literature. One of the remarkable refinements is due to G.~Navarro
and P.\,H.~Tiep \cite{Navarro-Tiep2}. They weakened the condition
that all nonlinear irreducible characters of $G$ have degree
divisible by $p$, and assumed only that those characters with values
in $\QQ_p$ have this property. (Here $\QQ_p$ is the cyclotomic field
obtained by adjoining a primitive $p$-root of unity to $\QQ$.) To
state their result, we write $\acd_{\FF,p'}$ to denote the average
of $p'$-degrees of irreducible characters of $G$ with values in a
field $\FF$.
\begin{theorem*}[Navarro and Tiep \cite{Navarro-Tiep2}]
Let $G$ be a finite group and $p$ a prime. If $\acd_{\QQ_p,p'}(G)=1$
then $G$ has a normal $p$-complement.
\end{theorem*}

We are able to prove the following.

\begin{theorem}\label{theorem-main-3}
Let $p$ be an odd prime and $G$ a finite group. We have
\begin{enumerate}
\item[(i)] if $\acd_{\QQ,2'}(G)<3/2$ then $G$ has a normal
$2$-complement, and
\item[(ii)] if $\acd_{\QQ_p,p'}(G)<4/3$ then $G$ has a normal
$p$-complement.
\end{enumerate}
\end{theorem}

\noindent Theorem~\ref{theorem-main-3} implies several earlier
results related to Thompson's theorem and fields of character
values, including \cite[Theorem~A]{Navarro-Sanus},
\cite[Theorem~A]{Navarro-Sanus-Tiep}, and \cite[Theorem~A and
Theorem~C]{Navarro-Tiep2}. Moreover, it has the following
consequence.

\begin{corollary}\label{corollary-main-4}
Let $p$ be an odd prime and $G$ a finite group. Then we have:
\begin{enumerate}
\item[(i)] If $\acd_{\QQ}(G)<3/2$ then $G$ has a normal
$2$-complement.
\item[(ii)] If $\acd_{\QQ_p}(G)<4/3$ then $G$ has a normal
$p$-complement.
\item[(iii)] If $\acd_{\RR,2'}(G)<3/2$ then $G$ has a normal $2$-complement.
\item[(iv)] If $\acd_{\RR}(G)<3/2$ then $G$ has a normal $2$-complement.
\end{enumerate}
\end{corollary}

\begin{proof}
The statements (i) and (ii) are clear from
Theorem~\ref{theorem-main-3}. Since every real-valued character of
degree 1 is also rational-valued, (iii) follows from
Theorem~\ref{theorem-main-3}(i), and finally, (iv) follows from (i)
or (iii).
\end{proof}

To prove the solvability, we utilize a character-orbit result on
nonabelian simple groups of Navarro and Tiep
\cite[Theorem~3.3]{Navarro-Tiep2} to show that, if $G$ has a
nonabelian minimal normal subgroup $N$, then there exists
$\psi\in\Irr(N)$ of large degree with many \emph{good} properties
such as $\psi(1)$ is coprime to $p$, $\psi$ is extendible to the
stabilizer $\Stab_G(\psi)$ of $\psi$ in $G$, and $|G:\Stab_G(\psi)|$
is coprime to $p$, see Theorem~\ref{theorem-psi-p-odd}. This,
together with other results on bounding the number of irreducible
characters of small degree in
Section~\ref{section-bounding-the-number}, allow us to control the
average of $p'$-degrees. To go from solvability to $p$-nilpotency,
we reduce the problem to the situation where $G$ is a split
extension of an abelian $p$-group, and then analyze the $\acd_{p'}$
of such a group. We hope that some new techniques in this paper will
be further developed to study other problems on the connection
between the average character degree, fields of character values,
and the local structure of groups, see \cite{Hung-Tiep} for
instance.


The paper is organized as follows. After some preparation results in
Sections~\ref{section-extending-characters},
\ref{section-bounding-the-number}, and
\ref{section-central-product}, we prove Theorem~\ref{theorem-main-2}
in Sections~\ref{section-p=2}, \ref{section-p=3}, \ref{section-p=5},
and \ref{section-p>5}. Theorem~\ref{theorem-main-1} is then proved
in Section~\ref{section-p-nilpotency}. In
Section~\ref{section-Qp-solvability} we establish some solvability
results on the average of $p'$-degrees of rational-valued characters
and $\QQ_p$-valued characters in general. Finally,
Theorem~\ref{theorem-main-3} is proved in
Section~\ref{section-Qp-nilpotency}.


\section{Extending characters of
$p'$-degree}\label{section-extending-characters}

We begin by setting up some notation. As usual, $\Irr(G)$ denotes
the set of irreducible characters of a finite group $G$, and
$\Irr_{p'}(G)$ the set of those characters of degree not divisible
by $p$. If $d$ is a positive integer, then $n_d(G)$ is the number of
irreducible characters of $G$ of degree $d$. If $N\trianglelefteq
G$, then
\[\Irr(G|N):=\{\chi\in\Irr(G)\mid N\not\subseteq\Ker(\chi)\},\]
\[\Irr_{p'}(G|N):=\{\chi\in\Irr(G)\mid N\not\subseteq\Ker(\chi), p\nmid\chi(1)\},\] and
\[n_d(G|N):=|\{\chi\in\Irr(G|N))\mid \chi(1)=d\}|.\] We also write
$\acd_{p'}(G|N)$ to denote the average degree of the characters in
$\Irr_{p'}(G|N)$. Furthermore, if $\theta\in\Irr(N)$ then
$\Irr_{p'}(G|\theta)$ denotes the set of irreducible characters of
degree coprime to $p$ of $G$ that lie over $\theta$, and
$\acd_{p'}(G|\theta)$ denotes the average degree of the characters
in $\Irr_{p'}(G|\theta)$. Finally, whenever a field $\FF$ is put
into the subscript of any of these notation, we mean that the
characters in consideration have values in $\FF$.

The following result plays an important role in the proof of
Theorem~\ref{theorem-main-2}. It helps us to bound the number of
irreducible characters of small degree in finite groups with a
nonabelian minimal normal subgroup.

\begin{theorem}\label{theorem-psi-p-odd}
Let $p$ be a prime. Let $G$ be a finite group with a nonabelian
minimal normal subgroup $N\ncong \Al_5$. Then there exists
$\psi\in\Irr(N)$ such that
\begin{enumerate}
\item[(i)] $\psi(1)\geq 7$ and $\psi(1)$ is coprime to $p$,
\item[(ii)] $\psi$ is extendible to a $\QQ_p$-valued character of $\Stab_G(\psi)$, and
\item[(iii)] $|G:\Stab_G(\psi)|$ is coprime to $p$.
\end{enumerate}
\end{theorem}

To prove this theorem, we need the following character-orbit result
for finite simple groups, which is essentially due to Navarro and
Tiep.

\begin{lemma}\label{lemma-orbit}
Let $p$ be a prime and $S$ be a nonabelian finite simple group. Then
there exists an orbit $\mathcal{\mathcal{O}}$ of the action of
$\Aut(S)$ on $\Irr(S)$ satisfying the following conditions:
\begin{enumerate}
\item[(i)] every $\theta\in \mathcal{O}$ is nontrivial of degree at least $4$ and coprime to $p$,
\item[(ii)] $|\mathcal{O}|$ is coprime to $p$, and
\item[(iii)] every $\theta\in \mathcal{O}$ extends to a
$\QQ_p$-valued character of $\Stab_{\Aut(S)}(\theta)$.
\end{enumerate}
Furthermore, if $S\ncong \Al_5$ then $\mathcal{O}$ can be chosen so
that $\theta(1)\geq 7$ for every $\theta\in\mathcal{O}$.
\end{lemma}

\begin{proof}
The orbit $\mathcal{O}$ has been constructed in
\cite[Theorem~3.3]{Navarro-Tiep2}, but without the condition that
$\theta(1)\geq 7$ when $S\ncong \Al_5$. The case $S\cong \Al_5$ is
clear from \cite[p.~2]{Atl1}.

Though by following the proof in \cite{Navarro-Tiep2} one can show
that $\theta(1)\geq 7$ when $S\ncong \Al_5$, we propose here another
way to verify it. First, by \cite{Rasala} the smallest nontrivial
degree of the alternating group $\Al_n$ is $n-1$ when $n\geq 6$.
Together with results on the low-degree characters of simple groups
of Lie type in \cite{Lubeck,Nguyen1,Tiep-Zalesskii} and the
character tables of the sporadic simple groups in \cite{Atl1}, we
can check that if a nonabelian simple group $S$ is not one of
$\Al_5$, $\Al_6$, $\Al_7$, $\PSL(2,7)$, $\PSL(2,11)$, and
$\PSU(3,3)$, then the smallest degree of a nontrivial irreducible
character of $S$ is at least $7$, and thus the condition
$\theta(1)\geq 7$ is automatically satisfied. For the exceptional
groups, the desired orbit can be found easily from \cite{Atl1}.
\end{proof}

Now we use the orbit $\mathcal{O}$ to prove
Theorem~\ref{theorem-psi-p-odd}.

\begin{proof}[Proof of Theorem \ref{theorem-psi-p-odd}] Since $N$ is
a nonabelian minimal normal subgroup of $G$, it is direct product of
$r$ copies of a nonabelian simple group, say $S$. Replacing $G$ by
$G/\bC_G(N)$ if necessary, we may assume that $\bC_G(N)=1$. Then we
have
\[
N\unlhd G \leq \Aut(N)=\Aut(S)\wr \Sy_r.
\]
Let $\theta$ be an irreducible character of $S$ in the orbit
$\mathcal{O}$ found in Lemma~\ref{lemma-orbit}. Consider the
character $\varphi:=\theta\times\cdots\times\theta\in\Irr(N)$. The
stabilizer of $\varphi$ in $\Aut(N)$ is
\[
\Stab_{\Aut(N)}(\varphi)=\Stab_{\Aut(S)}(\theta)\wr \Sy_r.
\]
By the choice of $\mathcal{O}$, $\theta$ extends to a $\QQ_p$-valued
character, say $\alpha$, of $\Stab_{\Aut(S)}(\theta)$. Thus
$\varphi$ extends to the character $\alpha\times\cdots\times\alpha$
of
$\Stab_{\Aut(S)}(\theta)\times\cdots\times\Stab_{\Aut(S)}(\theta)$,
which is the base group of the wreath product
$\Stab_{\Aut(S)}(\theta)\wr \Sy_r$. Since
$\alpha\times\cdots\times\alpha$ is invariant under
$\Stab_{\Aut(N)}(\varphi)$, it follows from
\cite[Lemma~1.3]{Mattarei} that $\alpha\times\cdots\times\alpha$ is
extendible to $\Stab_{\Aut(N)}(\varphi)$. We deduce that $\varphi$
is extendible to $\Stab_{\Aut(N)}(\varphi)$. Let
$\phi\in\Irr(\Stab_{\Aut(N)}(\varphi))$ be an extension of
$\varphi$. By the formula for character values given in
\cite[Lemma~1.3]{Mattarei}, we can choose $\phi$ so that it values
are contained in the field of values of $\alpha$. That is, $\phi$ is
$\QQ_p$-valued.

Now we consider the action of $G$ on the set
\[
\mathcal{C}:=\{\theta_1\times\cdots\times\theta_r\mid
\theta_i\in\mathcal{O}\}
\]
of irreducible characters of $N$. Since the cardinality of this set
is $|\mathcal{O}|^r$, which is not divisible by $p$ by the choice of
$\mathcal{O}$, there must be a $G$-orbit of length coprime to $p$.
Let $\psi\in\mathcal{C}$ be a character in such an orbit. We then
have that $|G:\Stab_G(\psi)|$ is coprime to $p$.

Note that $\Aut(N)$ acts transitively on $\mathcal{C}$. Therefore
there is some $x\in\Aut(N)$ such that $\psi=\varphi^x$, and hence
$\Stab_{\Aut(N)}(\psi)=\Stab_{\Aut(N)}(\varphi)^x$. As
$\phi\in\Irr(\Stab_{\Aut(N)}(\varphi))$ is an extension of
$\varphi$, we deduce that $\phi^x\in\Irr(\Stab_{\Aut(N)}(\psi))$ is
an extension of $\psi$ to $\Stab_{\Aut(N)}(\psi)$. In particular,
$\phi^x\downarrow_{\Stab_G(\psi)}$ is an extension of $\psi$ to
$\Stab_G(\psi)=G\cap \Stab_{\Aut(N)}(\psi)$. We observe that, as
$\phi$ is $\QQ_p$-valued, $\phi^x$ is $\QQ_p$-valued as well.
Finally, we note that $\psi(1)=\theta(1)^r$ is not divisible by $p$
and that $\psi(1)\geq 7$ if $(S,r)\neq (\Al_5,1)$.
\end{proof}


\section{Bounding the number of characters of small
degree}\label{section-bounding-the-number}

We begin the section with the following observation.

\begin{lemma}\label{lemma-n1-n2-n3}
Let $G$ be a finite group and $T\leq G$. Then
\begin{enumerate}
\item[(i)] $n_1(G)\leq n_1(T)|G:T|$,
\item[(ii)] $n_2(G)\leq n_2(T)|G:T|+\frac{1}{2}n_1(T)|G:T|$, and
\item[(iii)] $n_3(G)\leq n_3(T)|G:T|+\frac{1}{3}n_1(T)|G:T|$.
\end{enumerate}
\end{lemma}

\begin{proof}
(i) This is clear since $n_1(G)=|G:G'|$ and $n_1(T)=|T:T'|$.

(ii) Let $\chi\in\Irr(G)$ with $\chi(1)=2$. Take $\phi$ to be an
irreducible constituent of $\chi\downarrow_T$. Frobenius reciprocity
then implies that $\chi$ in turn is an irreducible constituent of
$\phi^G$. If $\phi(1)=2$ then, as $\phi^G(1)=2|G:T|$, there are at
most $|G:T|$ irreducible constituents of degree $2$ of $\phi^G$. We
deduce that there are at most $n_2(T)|G:T|$ irreducible characters
of degree $2$ of $G$ that arise as constituents of $\phi^G$ with
$\phi(1)=2$.

On the other hand, if $\phi(1)=1$ then, as $\phi^G(1)=|G:T|$, there
are at most $|G:T|/2$ irreducible constituents of degree $2$ of
$\phi^G$. As above, we deduce that there are at most $n_1(T)|G:T|/2$
irreducible characters of degree $2$ of $G$ that arise as
constituents of $\phi^G$ with $\phi(1)=1$. Now (ii) is proved.

(iii) This can be argued similarly as in (ii). Let $\chi\in\Irr(G)$
with $\chi(1)=3$. Then $\chi\downarrow_T$ is either ireeducible, or
a sum of three linear characters of $T$, or a sum of one linear
character and one irreducible character of degree $2$ of $T$. In
particular, if $\chi\downarrow_T$ is reducible then there is always
a linear constituent in $\chi\downarrow_T$. Now we see that there
are at most $n_3(T)|G:T|$ irreducible characters of degree $3$ of
$G$ that arise as constituents of $\phi^G$ with $\phi(1)=3$ and
there are at most $n_1(T)|G:T|/3$ irreducible characters of degree
$3$ of $G$ that arise as constituents of $\phi^G$ with $\phi(1)=1$.
The proof is complete.
\end{proof}

Lemma~\ref{lemma-n1-n2-n3} can help us to bound $n_1(G)$, $n_2(G)$,
and $n_3(G)$ in terms of the number of irreducible characters of
larger degree, especially in the case $G$ has a nonabelian minimal
normal subgroup.

\begin{proposition}\label{proposition-n1-n2}
Let $G$ be a finite group with a nonabelian minimal normal subgroup
$N$. Assume that there is some $\psi\in\Irr(N)$ such that $\psi$ is
extendible to $\Stab_G(\psi)$ and let $T:=\Stab_G(\psi)$ and
$a:=\psi(1)|G:T|$. We have
\begin{enumerate}
\item[(i)] $n_1(G)\leq n_a(G)|G:T|$, and
\item[(ii)] $n_2(G)\leq n_{2a}(G)|G:T|+\frac{1}{2}n_a(G)|G:T|$,
\end{enumerate}
Moreover, if $G=T$ then $n_2(G)\leq n_{2a}(G)$.
\end{proposition}

\begin{proof}
First, by Lemma~\ref{lemma-n1-n2-n3}(i) we have $n_1(G)\leq
n_1(T)|G:T|$. On the other hand, as $N=N'\subseteq T'$, $N$ is
contained in the kernel of every linear character of $T$ so that
$n_1(T)=n_1(T/N)$. It follows that
\[
n_1(G)\leq n_1(T/N)|G:T|.
\]

Recall that $\psi\in\Irr(N)$ is extendible to $T$ and so we let
$\chi\in\Irr(T)$ be an extension of $\psi$. Using Gallagher's
theorem and Clifford's theorem (see \cite[Corollary~6.17 and
Theorem~6.11]{Isaacs1}), we see that each linear character $\lambda$
of $T/N$ produces the irreducible character $\lambda\chi$ of $T$ of
degree $\psi(1)$, and this character in turn produces the
irreducible character $(\lambda\chi)^G$ of $G$ of degree
$(\lambda\chi)^G(1)=\psi(1)|G:T|=a$. It follows that \[n_1(T/N)\leq
n_{a}(G)\] and we therefore have
\[
n_1(G)\leq n_{a}(G)|G:T|,
\]
as claimed in (i).

We now prove (ii). By Lemma~\ref{lemma-n1-n2-n3}(ii) we have
$n_2(G)\leq n_2(T)|G:T|+\frac{1}{2}n_1(T)|G:T|$. Since we have
already proved that $n_1(T)=n_1(T/N)\leq n_a(G)$, it remains to
prove that $n_2(T)\leq n_{2a}(G)$.

We claim that $n_2(T)=n_2(T/N)$ or in other words $N$ is contained
in the kernel of every irreducible character of degree $2$ of $T$.
Let $\phi\in\Irr(T)$ with $\phi(1)=2$. Since $N$ has no irreducible
character of degree 2 (see Problem~3.3 of~\cite{Isaacs1}) and has
only one linear character, which is the trivial one, it follows that
$\phi_N=2\cdot1_N$. We then have $N\subseteq \Ker(\phi)$, as
claimed.

Recall that $\chi\in\Irr(T)$ is an extension of $\psi$. Using
Gallagher's theorem and Clifford's theorem again, we obtain that
each irreducible character $\mu\in\Irr(T/N)$ of degree $2$ produces
the character $(\mu\chi)^G\in\Irr(G)$ of degree
$(\mu\chi)^G(1)=2\psi(1)|G:T|=2a$. It follows that
\[n_2(T/N)\leq n_{2a}(G),\] and thus $n_2(T)\leq n_{2a}(G)$, as wanted.
\end{proof}

\begin{proposition}\label{proposition-n3}
Let $G$ be a finite group with a nonabelian minimal normal subgroup
$N$, which has no direct factor isomorphic to $\Al_5$ or
$\PSL(2,7)$. Assume that there is some $\psi\in\Irr(G)$ such that
$\psi$ is extendible to $\Stab_G(\psi)$ and let $T:=\Stab_G(\psi)$
and $a:=\psi(1)|G:T|$. Then \[n_3(G)\leq
n_{3a}(G)|G:T|+\frac{1}{3}n_a(G)|G:T|.\] Moreover, if $G=T$ then
$n_3(G)\leq n_{3a}(G)$.
\end{proposition}

\begin{proof}
By Lemma~\ref{lemma-n1-n2-n3}(iii) we have $n_3(G)\leq
n_3(T)|G:T|+\frac{1}{3}n_1(T)|G:T|$. As in the proof of
Proposition~\ref{proposition-n1-n2}, we have $n_1(T)\leq n_a(G)$. So
it suffices to show that $n_3(T)\leq n_{3a}(G)$.

It is well known that $\Al_5$ and $\PSL(2,7)$ are the only finite
simple groups having an irreducible character of degree $3$.
Therefore, every nontrivial irreducible character of $N$ has degree
at least 4 and, by using the same arguments as in the proof of
Proposition~\ref{proposition-n1-n2}, we see that $N$ is contained in
the kernel of every irreducible character of degree 3 of $T$. In
other words we have $n_3(T)=n_3(T/N)$.

Recall that $\psi$ is extendible to $T$ and let $\chi\in\Irr(T)$ be
an extension of $\psi$. We then obtain an injection $\nu\mapsto
(\nu\chi)^G$ from the set of irreducible characters of $T/N$ of
degree $3$ to the set of irreducible characters of $G$ of degree
$(\nu\chi)^G(1)=3\psi(1)|G:T|=3a$. It follows that $n_3(T/N)\leq
n_{3a}(G),$ and therefore $n_3(T)\leq n_{3a}(G)$, which completes
the proof.
\end{proof}


\section{Characters of a central
product}\label{section-central-product}

The proof of Theorem~\ref{theorem-main-2} requires us to analyze the
characters of a particular central product. This central product
indeed has already appeared in the study of the average of
\emph{all} irreducible character degrees of a finite group,
see~\cite{Isaacs-Loukaki-Moreto,Moreto-Nguyen}.

\begin{proposition}\label{proposition-central-product}
Let $L\cong \SL(2,5)$ and $G=LC$ be a central product with the
central amalgamated subgroup $Z:=\bZ(L)=L\cap C$ such that
$L\subseteq G'$. Assume that $G$ has an irreducible character of
degree 2 such that $Z\nsubseteq \Ker(\chi)$. Then
\begin{enumerate}
\item[(i)] $n_2(G)=2n_1(G)+n_2(C/Z)$,
\item[(ii)] $n_3(G)\geq 2n_1(G)$,
\item[(iii)] $n_4(G)\geq 2n_1(G)$,
\item[(iv)] $n_5(G)\geq n_1(G)$,
\item[(v)] $n_6(G)\geq n_1(G)$, and
\item[(vi)] $n_8(G)\geq n_2(C/N)$.
\end{enumerate}
\end{proposition}

\begin{proof}
Since $G=LC$ is a central product with the central amalgamated
subgroup $Z$, there is a bijection $(\alpha,\beta)\mapsto \tau$ from
$\Irr(L|Z)\times\Irr(C|Z)$ to $\Irr(G|Z)$ such that
$\tau(1)=\alpha(1)\beta(1)$.

By hypothesis, there is $\chi\in\Irr(G|Z)$ such that $\chi(1)=2$. If
$(\alpha,\beta)\mapsto \chi$ under the above bijection, we must have
$\beta(1)=1$ since $L\cong\SL(2,5)$ and there are only three
possibilities for $\alpha(1)$, namely 2, 4, and 6. So
$\beta\in\Irr(C|Z)$ is an extension of the unique nonprincipal
linear character of $Z$. Using Gallagher's theorem, we then have a
degree-preserving bijection from $\Irr(C/Z)$ to $\Irr(C|Z)$. In
particular,
\[n_1(C|Z)=n_1(C/Z)\]
Since $G/L\cong C/Z$ and $L\subseteq G'$, we have
\[n_1(C|Z)=n_1(C/Z)=n_1(G).\]

Now we evaluate $n_2(G)$. As $n_1(L|Z)=0$ and $n_2(L|Z)=2$, we have
\[n_2(G|Z)=n_1(L|Z)n_2(C|Z)+n_2(L|Z)n_1(C|Z)=2n_1(C|Z)=2n_1(G).\]
Note that there is also a bijection $(\alpha,\beta)\mapsto \tau$
from $\Irr(L/Z)\times\Irr(C/Z)$ to $\Irr(G/Z)$ such that
$\tau(1)=\alpha(1)\beta(1)$. Therefore, as $n_1(L/Z)=1$ and
$n_2(L/Z)=0$, we have
\[n_2(G/Z)=n_1(L/Z)n_2(C/Z)+n_2(L/Z)n_1(C/Z)=n_2(C/Z),\] and it follows
that
\[n_2(G)=n_2(G/Z)+n_2(G|Z)=n_2(C/Z)+2n_1(G).\]

Next we estimate $n_4(G)$, $n_5(G)$, and $n_6(G)$. We have
\[n_4(G|Z)\geq n_4(L|Z)n_1(C|Z)\geq n_1(C|Z)=n_1(G)\] since
$n_4(L|Z)=1$ and
\[n_4(G/Z)\geq n_4(L/Z)n_1(C/Z)=n_1(C/Z)=n_1(G)\] since
$n_4(L/Z)=1$.
We deduce that
\[n_4(G)=n_4(G|Z)+n_4(G/Z)\geq 2n_1(G).\]
Similarly,
\[n_5(G)\geq n_5(G/Z)\geq a_5(L/Z)a_1(C/Z)=a_1(C/Z)=n_1(G)\] since $n_5(L/Z)=1$, and
\[n_6(G)\geq n_6(G|Z)\geq n_6(L|Z)n_1(C|Z)=n_1(G)\] since
$n_6(L|Z)=1$.

Finally, we estimate $n_8(G)$ by
\[n_8(G)\geq n_8(G/N)\geq n_4(L/N)n_2(C/N)\geq n_2(C/N),\]
and we have completed the proof.
\end{proof}


\section{Characters of odd degree and
solvability}\label{section-p=2}

In this section we prove Theorem~\ref{theorem-main-2}(i), which we
restate below for the reader's convenience.

\begin{theorem}\label{theorem-odd<3-then-G-solvable}
Let $G$ be a finite group. If $\acd_{2'}(G)<3$ then $G$ is solvable.
\end{theorem}

\begin{proof}
Assume that the theorem is false, and let $G$ be a minimal
counterexample. In particular, $G$ is nonsolvable and
$\acd_{2'}(G)<3$. Then we have
\[
\frac{\sum_{d~\odd} dn_d(G)}{\sum_{d~\odd} n_d(G)}<3,
\]
and hence
\[
\sum_{d\geq 5~\odd}(d-3)n_d(G)<2n_1(G).
\]
Since $G$ is nonsolvable, $G'$ is nontrivial and therefore we can
choose a minimal normal subgroup $N$ of $G$ such that $N\subseteq
G'$. So $N$ is contained in the kernel of every linear character of
$G$ so that $n_1(G)=n_1(G/N)$. Therefore
\[ \sum_{d\geq 5~\odd}(d-3)n_d(G/N)\leq
\sum_{d\geq 5~\odd}(d-3)n_d(G)<2n_1(G)=2n_1(G/N)
\]
and it follows that
\[\acd_{2'}(G/N)<3.
\]
By the minimality of $G$, we deduce that $G/N$ is solvable. But $G$
is nonsolvable, so $N$ is a nonabelian minimal normal subgroup of
$G$. Theorem~\ref{theorem-psi-p-odd} then implies that $N$ has an
irreducible character $\psi$ with three properties:
\begin{enumerate}
\item[(i)] $\psi(1)\geq 5$ is odd,
\item[(ii)] $\psi$ is extendible to $\Stab_G(\psi)$, and
\item[(iii)] $|G:\Stab_G(\psi)|$ is odd.
\end{enumerate}
(In Theorem~\ref{theorem-psi-p-odd} we in fact assume that
$N\ncong\Al_5$. But one easily sees that if $N\cong \Al_5$ then the
character $\psi$ can be chosen to be the unique irreducible
character of degree 5 of $N$.)

Now applying Proposition~\ref{proposition-n1-n2}(i), we have
\[
n_1(G)\leq n_a(G)|G:T|,
\]
where $T:=\Stab_G(\psi)$ and $a:=\psi(1)|G:T|$. As $\psi(1)\geq 5$,
we have $a\geq 5|G:T|$ and it follows that
\[
n_1(G)\leq \frac{1}{2}n_{a}(G) (a-3).
\]
Using the fact that $a=\psi(1)|G:T|$ is odd by the choice of $\psi$,
we arrive at
\[
n_1(G)\leq \frac{1}{2}\sum_{d\geq 5~\odd}(d-3)n_d(G)
\]
and, equivalently, $\acd_{2'}(G)\geq 3$. This contradiction
completes the proof.
\end{proof}


\section{Characters of $3'$-degree and solvability}\label{section-p=3}

In this section we prove Theorem~\ref{theorem-main-2}(ii).

\begin{theorem}\label{theorem-3'<3-then-G-solvable}
Let $G$ be a finite group. If $\acd_{3'}(G)<3$ then $G$ is solvable.
\end{theorem}

First we handle the groups with a nonabelian minimal normal
subgroup.

\begin{proposition}\label{proposition-p=3}
Let $G$ be a finite group with a nonabelian minimal normal subgroup
$N$. Then $\acd_{p'}(G)\geq 3$.
\end{proposition}

\begin{proof}
Suppose that $N$ is direct product of $r$ copies of $S$, a
nonabelian simple group. First we consider the case $N\cong \Al_5$.
Then $N$ has an irreducible character of degree $5$ that is
extendible to $G$. Applying Proposition~\ref{proposition-n1-n2}, we
have
\[n_1(G)\leq n_5(G)\]
and \[n_2(G)\leq n_{10}(G).\] It follows that \[2n_1(G)+n_2(G)\leq
2n_5(G)+n_{10}(G)\] and hence
\[2n_1(G)+n_2(G)\leq \sum_{3\nmid d, d\geq 4} (d-3)n_d(G),\]
which is equivalent to $\acd_{3'}(G)\geq 3$, and we are done.

So we may assume that $N\ncong \Al_5$. By
Theorem~\ref{theorem-psi-p-odd}, there is some $\psi\in\Irr(N)$ with
the conditions:
\begin{enumerate}
\item[(i)] $\psi(1)\geq 7$ and $3\nmid\psi(1)$,
\item[(ii)] $\psi$ is extendible to $\Stab_G(\psi)$, and
\item[(iii)] $3\nmid|G:\Stab_G(\psi)|$.
\end{enumerate}
We then apply Proposition~\ref{proposition-n1-n2} to have
\[n_1(G)\leq n_a(G)|G:T|\]
and
\[n_2(G)\leq n_{2a}(G)|G:T|+\frac{1}{2}n_a(G)|G:T|,\]
where $T:=\Stab_G(\psi)$ and $a:=\psi(1)|G:T|$. It then follows that
\[2n_1(G)+n_2(G)\leq
\frac{5}{2}n_a(G)|G:T|+n_{2a}(G)|G:T|.\] Note that $\psi(1)\geq 7$,
and thus $a\geq 7|G:T|$. So we have $(5/2)|G:T|<a-3$ and
$|G:T|<2a-3$. Therefore
\[2n_1(G)+n_2(G)<
(a-3)n_a(G)+(2a-3)n_{2a}(G).\] As $a$ is coprime to $3$, we deduce
that
\[2n_1(G)+n_2(G)< \sum_{3\nmid d, d\geq 4} (d-3)n_d(G),\]
which is equivalent to $\acd_{3'}(G)> 3$, and we are done again.
\end{proof}

Now we are ready to prove
Theorem~\ref{theorem-3'<3-then-G-solvable}. We write $\bO_\infty(G)$
to denote the largest solvable normal subgroup of $G$.

\begin{proof}[Proof of Theorem~\ref{theorem-3'<3-then-G-solvable}]
Assume that the theorem is false and let $G$ be a minimal
counterexample. Then $G$ is nonsolvable and $\acd_{3'}(G)<3$.

Let $L\lhd G$ be minimal such that $L$ is non-solvable. Then clearly
$L$ is perfect and contained in the last term of the derived series
of $G$. Let $N\subseteq L$ be a minimal normal subgroup of $G$. We
choose $N$ so that $N\leq [L,\bO_{\infty}(L)]$ if
$[L,\bO_{\infty}(L)]$ is nontrivial. We then have $N \subseteq
L=L'\subseteq G'$.

If $N$ is nonabelian then $\acd_{3'}(G)\geq 3$ by
Proposition~\ref{proposition-p=3}, and this is a contradiction. So
we may assume that $N$ is abelian so that $G/N$ is nonsolvable. By
the minimality of $G$, it follows that $\acd_{3'}(G/N)\geq 3$ and
hence
\[\acd_{3'}(G)<3\leq\acd_{3'}(G/N).\]
Note that $n_d(G)\geq n_d(G/N)$ for every positive integer $d$ and
$n_1(G)=n_1(G/N)$ since $N\subseteq G'$. We then deduce that
\[n_2(G)>n_2(G/N).\]
That is, there is some $\chi\in\Irr(G)$ of degree $2$ whose kernel
does not contain $N$.

Now let $C/\Ker(\chi):=\bZ(G/\Ker(\chi))$. Arguing similarly as in
the proof of \cite[Theorem~2.2]{Isaacs-Loukaki-Moreto}, we obtain
that $G/C\cong \Al_5$, $L\cong \SL(2,5)$, and $G=LC$ is a central
product with the central amalgamated subgroup $Z:=L\cap C=\bZ(L)$.

We are now in the situation of
Proposition~\ref{proposition-central-product}. Therefore
\begin{align*}
2n_1(G)+n_2(G)&=4n_1(G)+n_2(C/Z)\\
&\leq n_4(G)+2n_5(G)+n_8(G)\\
&\leq\sum_{3\nmid d, d\geq 4}(d-3)n_d(G).
\end{align*}
It then follows that $\acd_{3'}(G)\geq 3$ and this is a
contradiction.
\end{proof}


\section{Characters of $5'$-degree and solvability}\label{section-p=5}

In this section we prove Theorem~\ref{theorem-main-2}(iii).

\begin{theorem}\label{theorem-5'<11/4-then-G-solvable}
Let $G$ be a finite group. If $\acd_{5'}(G)<11/4$ then $G$ is
solvable.
\end{theorem}

As in Section~\ref{section-p=3}, we first handle finite groups with
a nonabelian minimal normal subgroup.

\begin{proposition}\label{proposition-p=5}
Let $G$ be a finite group with a nonabelian minimal normal subgroup
$N$. Then $\acd_{5'}(G)\geq 11/4$.
\end{proposition}

\begin{proof}
As before, we suppose that $N$ is direct product of $r$ copies of a
nonabelian simple group $S$. First we consider $N\cong \Al_5$. Then
$N$ has an irreducible character of degree $4$ that is extendible to
$G$. Applying Proposition~\ref{proposition-n1-n2}, we have
$n_1(G)\leq n_4(G)$ and $n_2(G)\leq n_{8}(G)$.

Now we need to estimate $n_3(G)$ and $n_6(G)$. Observe that $N$ has
two irreducible characters of degree $3$ and let us denote them by
$\psi_1$ and $\psi_2$. Then it is easy to see that both $\psi_1$ and
$\psi_2$ are extendible to
\[T:=\Stab_G(\psi_1)=\Stab_G(\psi_2)=N\times\bC_G(N).\]
Since $N$ has index 2 in $\Aut(N)=\Sy_5$, we have
\[|G:T|=1 \text{ or } 2.\]

If $|G:T|=1$ then each linear character of $G$, which can be
considered as a linear character of $G/N$, produces two irreducible
characters of $G$ of degree $3$, one lying above $\psi_1$ and the
other lying above $\psi_2$. We then obtain that $2n_1(G)\leq
n_3(G)$. Now taking $n_1(G)\leq n_4(G)$ and $n_2(G)\leq n_{8}(G)$
into account, we have
\[7n_1(G)+3n_2(G)\leq n_3(G)+5n_4(G)+3n_8(G).\]
Therefore \[7n_1(G)+3n_2(G)\leq \sum_{5\nmid d, d\geq
3}(4d-11)n_d(G),\] and thus $\acd_{5'}(G)\geq 11/4$, as desired.

If $|G:T|=2$ then by Proposition~\ref{proposition-n1-n2}(1) we have
$n_1(G)\leq 2n_6(G)$. Similarly we have
\[7n_1(G)+3n_2(G)\leq 5n_4(G)+3n_8(G)+4n_6(G).\]
Therefore \[7n_1(G)+3n_2(G)\leq \sum_{5\nmid d, d\geq
3}(4d-11)n_d(G),\] and we are done again.

From now on to the end of the proof we can assume that $N\ncong
\Al_5$ and we will argue as in the proof of
Proposition~\ref{proposition-p=3}. By
Theorem~\ref{theorem-psi-p-odd}, there is some $\psi\in\Irr(N)$ such
that $\psi(1)\geq 7$, $5\nmid\psi(1)$, $\psi$ is extendible to
$\Stab_G(\psi)$, and $5\nmid |G:\Stab_G(\psi)|$.

We then apply Proposition~\ref{proposition-n1-n2} to have
\[n_1(G)\leq n_a(G)|G:T|\]
and
\[n_2(G)\leq n_{2a}(G)|G:T|+\frac{1}{2}n_a(G)|G:T|,\]
where $T:=\Stab_G(\psi)$ and $a:=\psi(1)|G:T|$. It then follows that
\[7n_1(G)+3n_2(G)\leq
\frac{17}{2}n_a(G)|G:T|+3n_{2a}(G)|G:T|.\] Since $\psi(1)\geq 7$, we
have $a\geq 7|G:T|$, and therefore $(17/2)|G:T|<4a-11$ and
$3|G:T|<8a-11$. We deduce that
\[7n_1(G)+3n_2(G)<
(4a-11)n_a(G)+(8a-11)n_{2a}(G),\] and it follows that $\acd_{5'}(G)>
11/4$. The proof is complete.
\end{proof}

\begin{proof}[Proof of Theorem~\ref{theorem-5'<11/4-then-G-solvable}]
Assume that the theorem is false and let $G$ be a minimal
counterexample. Then $G$ is nonsolvable and $\acd_{5'}(G)<11/4$.

By using Proposition~\ref{proposition-p=5} and choosing the
subgroups $L,N$, and $C$ as in the proof of
Theorem~\ref{theorem-3'<3-then-G-solvable}, we have that $G/C\cong
\Al_5$, $L\cong \SL(2,5)$, and $G=LC$ is a central product with the
central amalgamated subgroup $Z:=L\cap C=\bZ(L)$.

Applying Proposition~\ref{proposition-central-product}, we deduce
that
\begin{align*}
7n_1(G)+3n_2(G)&=13n_1(G)+3n_2(C/Z)\\
&\leq n_3(G)+5n_4(G)+n_6(G)+3n_8(G)\\
&\leq\sum_{5\nmid d, d\geq 3}(4d-11)n_d(G).
\end{align*}
From this it follows that $\acd_{5'}(G)\geq 11/4$, violating the
assumption.
\end{proof}


\section{Characters of $p'$-degree for $p>5$ and
solvability}\label{section-p>5}

We now prove Theorem~\ref{theorem-main-2}(iv) and therefore complete
the proof of Theorem~\ref{theorem-main-2}.

\begin{theorem}\label{theorem-p'<16/5then-G-solvable}
Let $p>5$ be a prime and $G$ and finite group. If
$\acd_{p'}(G)<16/5$ then $G$ is solvable.
\end{theorem}

Unlike the proofs for the smaller primes, the proof of
Theorem~\ref{theorem-p'<16/5then-G-solvable} requires upper bounds
for not only $n_1(G)$ and $n_2(G)$ but also $n_3(G)$ and this makes
things harder. Since the simple groups $\Al_5$ and $\PSL(2,7)$ have
some irreducible characters of degree 3, they need special
attention.

\begin{lemma}\label{lemma-A5-PSL27}
Let $G$ be a finite group with a minimal normal subgroup $N$, which
is direct product of $r$ copies of a nonabelian simple group $S$. We
have
\begin{enumerate}
\item[(i)] if $S\cong \Al_5$ then $N$ has two irreducible characters
of degrees $4^r$ and $5^r$ which are both extendible to $G$; and
\item[(ii)] if $S\cong \PSL(2,7)$ then $N$ has three irreducible
characters of degrees $6^r,7^r,8^r$ which are all extendible to $G$.
\end{enumerate}
\end{lemma}

\begin{proof}
This is almost obvious as $\Al_5$ has two irreducible characters of
degrees 4 and 5 which are both extendible to $\Aut(\Al_5)=\Sy_5$,
and $\PSL(2,7)$ has three irreducible characters of degrees $6, 7,
8$, which are all extendible to $\Aut(\PSL(2,7))=\PGL(2,7)$.
\end{proof}

Using the previous lemma and the techniques in the proofs of
Propositions~\ref{proposition-n1-n2} and \ref{proposition-n3}, we
have the following.

\begin{proposition}\label{proposition-A5-PSL27}
Let $G$ be a finite group with a nonabelian minimal normal subgroup
$N$, which is direct product of $r$ copies of a simple group $S$. We
have
\begin{enumerate}
\item[(i)] if $S\cong \Al_5$ then $n_1(G)\leq \min\{n_{4^r}(G),n_{5^r}(G)\}$, $n_2(G)\leq n_{2\cdot 5^r}(G)$, and $n_3(G)\leq n_{3\cdot
5^r}(G)+2rn_1(G)$; and
\item[(ii)] if $S\cong \PSL(2,7)$ then $n_1(G)\leq n_{8^r}(G)$, $n_2(G)\leq n_{2\cdot 8^r}(G)$, and $n_3(G)\leq n_{3\cdot
8^r}(G)+2rn_1(G)$.
\end{enumerate}
\end{proposition}

\begin{proof}
The proofs of (i) and (ii) are fairly similar, so let us prove (i)
only. So assume that $S\cong \Al_5$. Indeed, the inequalities
$n_1(G)\leq \min\{n_{4^r}(G),n_{5^r}(G)\}$ and $n_2(G)\leq n_{2\cdot
5^r}(G)$ already follows from Proposition~\ref{proposition-n1-n2}
and hence it remains to prove $n_3(G)\leq n_{3\cdot
5^r}(G)+2rn_1(G)$.

Since $N$ has an irreducible character of degree $5^r$ that is
extendible to $G$, Gallagher's theorem implies that there is an
injection from the irreducible characters of degree 3 of $G/N$ to
the irreducible characters of degree $3\cdot 5^r$ of $G$ . That is
\[
n_3(G/N)\leq n_{3\cdot 5^r}(G).
\]

Now we need to bound the number of irreducible characters of $G$ of
degree 3 whose kernels do not contain $N$. So let $\chi\in\Irr(G)$
such that $\chi(1)=3$ and $N\nsubseteq \Ker(\chi)$. Since $N$ has no
nonprincipal linear character and no irreducible character of degree
2, the restriction $\chi\downarrow_N$ must be irreducible. By
Gallagher's theorem, the number of irreducible characters of $G$ of
degree 3 lying over $\chi\downarrow_N$ equals to $n_1(G/N)$, which
is the same as $n_1(G)$. Note that $\chi\downarrow_N$ has degree 3
and $N$ has exactly $2r$ irreducible characters of degree 3. We
conclude that the number of irreducible characters of $G$ of degree
3 whose kernels do not contain $N$ is at most $2kn_1(G)$. Now we
have $n_3(G|N)\leq 2rn_1(G)$ and thus
\[
n_3(G)=n_3(G/N)+n_3(G|N)\leq n_{3\cdot 5^r}(G)+2rn_1(G),
\]
as desired.
\end{proof}

The next result is a refinement of
\cite[Proposition~3]{Moreto-Nguyen} for characters of $p'$-degrees.

\begin{proposition}\label{proposition-p>5-1}
Let $p>5$ be a prime. Let $G$ be a finite group with a nonabelian
minimal normal subgroup $N$. Then $\acd_{p'}(G)\geq 16/5$.
\end{proposition}

\begin{proof}
Suppose that $N$ is direct product of $r$ copies of $S$, a
nonabelian simple group.

First we assume that $S\ncong\Al_5$ and $S\ncong\PSL(2,7)$. It then
follows from Theorem~\ref{theorem-psi-p-odd} that there is some
$\psi\in\Irr(N)$ such that $\psi(1)\geq 7$, $\psi(1)$ is coprime to
$p$, $\psi$ is extendible to $\Stab_G(\psi)$, and
$|G:\Stab_G(\psi)|$ is coprime to $p$.
Propositions~\ref{proposition-n1-n2} and \ref{proposition-n3} then
imply that
\[n_1(G)\leq n_a(G)|G:T|,\]
\[n_2(G)\leq n_{2a}(G)|G:T|+\frac{1}{2}n_a(G)|G:T|,\]
and \[n_3(G)\leq n_{3a}(G)|G:T|+\frac{1}{3}n_a(G)|G:T|,\] where
$T:=\Stab_G(\psi)$ and $a:=\psi(1)|G:T|$. Now we can estimate
\begin{align*}11n_1(G)+6n_2(G)+n_3(G)&\leq
\frac{43}{3}n_a(G)|G:T|+6n_{2a}(G)|G:T|+n_{3a}(G)|G:T|\\
&<
(5a-16)n_a(G)+(10a-16)n_{2a}(G)+(15a-16)n_{3a}(G)\\
&\leq \sum_{p\nmid d, d\geq 4}(5d-16)n_d(G),\end{align*} where the
last two inequalities follow from the fact that $a\geq 7|G:T|$ and
$a$ is coprime to $p>5$. Now it follows that $\acd_{p'}(G)> 16/5$
and we are done.

Next we consider the case $S\cong \Al_5$. We use
Proposition~\ref{proposition-A5-PSL27}(1) to deduce that
\begin{align*}11n_1(G)+6n_2(G)+n_3(G)&\leq
9n_{5^r}(G)+6n_{2\cdot 5^r}(G)+n_{3\cdot 5^r}(G)+(2+2r)n_{4^r}(G)\\
&\leq\begin{aligned}
&(5\cdot 5^r-16)n_{5^r}(G)+(10\cdot5^r-16)n_{2\cdot 5^r}(G)\\&+(15\cdot 5^r-16)n_{3\cdot 5^r}(G)+(5\cdot 4^r-16)n_{4^r}(G)\end{aligned}\\
&\leq \sum_{p\nmid d, d\geq 4}(5d-16)n_d(G),\end{align*} and we are
done again. The case $S\cong \PSL(2,7)$ is treated similarly with
the help of Proposition~\ref{proposition-A5-PSL27}(2) and we skip
the details.
\end{proof}

We are now able to prove
Theorem~\ref{theorem-p'<16/5then-G-solvable}.

\begin{proof}[Proof of Theorem~\ref{theorem-p'<16/5then-G-solvable}]
Assume, to the contrary, that the theorem is false and let $G$ be a
minimal counterexample. Then $G$ is nonsolvable and
$\acd_{p'}(G)<16/5$.

As in the proof of Theorem~\ref{theorem-3'<3-then-G-solvable}, we
let $L\lhd G$ be minimal such that $L$ is non-solvable and let
$N\subseteq L$ be a minimal normal subgroup of $G$. We choose $N$
such that $N\leq [L,\bO_{\infty}(L)]$ if $[L,\bO_{\infty}(L)]$ is
nontrivial and when possible we choose $N$ to be of order 2. Note
that $L$ is perfect and $L\subseteq G'$.

If $N$ is nonabelian then $\acd_{p'}(G)\geq 16/5$ by
Proposition~\ref{proposition-p>5-1} and so we are done. Therefore we
assume from now on that $N$ is abelian. As $G$ is nonsolvable, it
follows that so is $G/N$. By the minimality of $G$, we then have
$\acd_{p'}(G/N)\geq 16/5$ and hence
\[\acd_{p'}(G)<16/5\leq\acd_{p'}(G/N).\]
Since $n_1(G)=n_1(G/N)$ as $N\subseteq L\subseteq G'$, we then
deduce that
\[\text{either }n_2(G)>n_2(G/N) \text{ or } n_3(G)>n_3(G/N).\]
That is, there is some irreducible character $\chi\in\Irr(G)$ of
degree $2$ or 3 such that $\Ker(\chi)$ does not contain $N$.

Now we can use the classification of the primitive linear groups of
degree 2 and 3 in \cite[Chapter~V, Section~81]{bli} and argue
similarly as in the proof of \cite[Theorem~A]{Moreto-Nguyen} to
obtain that $G=LC$ is a central product with the central amalgamated
subgroup $L\cap C=\bZ(L)$, where $\bZ(L)\supseteq N>1$,
\[C/\Ker(\chi)=\bZ(G/\Ker(\chi))\] and \[L/\bZ(L)\cong G/C\cong \Al_5, \Al_6, \text{
or } \PSL(2,7).\] Moreover, from the proof of
Theorem~\ref{theorem-3'<3-then-G-solvable} we see that if
$\chi(1)=2$ then $L/\bZ(L)$ must be isomorphic to $\Al_5$.

Since $G=LC$ is a central product with the central amalgamated
subgroup $\bZ(L)$, for each $\lambda\in\Irr(\bZ(L))$ there is a
bijection
\[\Irr(L|\lambda)\times \Irr(C|\lambda)\rightarrow \Irr(G|\lambda)\]
such that if $(\alpha,\beta)\mapsto \chi$ then
$\chi(1)=\alpha(1)\beta(1)$. It is clear that $\chi(1)$ is coprime
to $p$ if and only if both $\alpha(1)$ and $\beta(1)$ are coprime to
$p$. Therefore this bijection produces another bijection
\[\Irr_{p'}(L|\lambda)\times
\Irr_{p'}(C|\lambda)\rightarrow \Irr_{p'}(G|\lambda)\] and in
particular we have
\[\acd_{p'}(G|\lambda)= \acd_{p'}(L|\lambda)\acd_{p'}(C|\lambda).\]
Therefore
\[\acd_{p'}(G|\lambda)\geq \acd_{p'}(L|\lambda).\]

If $L/\bZ(L)\cong\Al_5$ then we must have $L\cong \SL(2,5)$ since
this is the only nontrivial perfect central cover of $\Al_5$. So
$\bZ(L)\cong C_2$, the cyclic group of order 2. Now, using
\cite[p.~2]{Atl1} we can check that $\acd_{p'}(L|\lambda)\geq 16/5$
whether $\lambda$ is trivial or the only nontrivial character of
$\bZ(L)$. Thus $\acd_{p'}(G)\geq 16/5$ and we are done.

If $L/\bZ(L)\cong \PSL(2,7)$ then similarly we have $L\cong\SL(2,7)$
so that $\bZ(L)\cong C_2$. Since $N\subseteq L$ and $N\nsubseteq
\Ker(\chi)$, it follows that $L\nsubseteq \Ker(\chi)$. As
$\chi(1)=3$ and the smallest degree of a nontrivial irreducible
character of $L$ is 3, we deduce that the restriction
$\chi\downarrow_L\in\Irr(L)$. But then the character table of
$\SL(2,7)$ (see \cite[p. 3]{Atl1}) implies that $\bZ(L)\subseteq
\Ker(\chi\downarrow_L)$, which in turns implies that $N\subseteq
\Ker(\chi)$ since $N\subseteq \bZ(L)$, and this violates the choice
of $\chi$.

Finally we consider $L/\bZ(L)\cong \Al_6$. Then as mentioned above
we must have $\chi(1)=3$. Also, $L$ is one of three perfect central
covers of $\Al_6$, namely $2\cdot \Al_6$, $3\cdot \Al_6$, and
$6\cdot\Al_6$. First assume that $L\cong 2\cdot \Al_6$ or
$6\cdot\Al_6$. Then $N\cong C_2$ since we chose $N$ to be of order 2
when possible. Arguing as in the case $L/\bZ(L)\cong \PSL(2,7)$, we
obtain that $N$ is contained in the kernel of an irreducible
character of degree 3 of $6\cdot \Al_6$, and this is a contradiction
by \cite[p.~5]{Atl1}. So it remains to consider $L\cong
3\cdot\Al_6$. But then one can check that
$\acd_{p'}(L|\lambda)>16/5$ whether $\lambda$ is the trivial
character or one of the two nontrivial irreducible characters of
$\bZ(L)$. It follows that $\acd_{p'}(G|\lambda)>16/5$ for every
$\lambda\in\Irr(\bZ(L))$, and hence $\acd_{p'}(G)>16/5$ in this
case.
\end{proof}


\section{Characters of $p'$-degree and
$p$-nilpotency}\label{section-p-nilpotency}

We recall that, for a finite group $G$,
\[\Irr_{p'}(G):=\{\chi\in\Irr(G)\mid p\nmid \chi(1)\}.\] We begin
the section with the following easy observation, which can be viewed
as a $p'$-version of \cite[Lemma~3.1]{Isaacs-Loukaki-Moreto}.

\begin{lemma}\label{lemma-compare}
Let $p$ be a prime and $A$ be a subgroup of a finite group $G$. Then
\[|\Irr_{p'}(G)|\leq |G:A||\Irr_{p'}(A)|.\]
\end{lemma}

\begin{proof}
Let $\chi$ be an irreducible character of $G$ such that $\chi(1)$ is
not divisible by $p$. Consider the restriction $\chi_A$. There must
be an irreducible constituent $\lambda\in\Irr(A)$ of $\chi_A$ such
that $\lambda(1)$ is not divisible by $p$, and moreover, $\chi$ in
turn is an irreducible constituent of $\lambda^G$ by Frobenius
reciprocity. On the other hand, given any $\lambda\in\Irr(A)$, each
irreducible constituent of $\lambda^G$ has degree at least
$\lambda(1)$, and therefore the number of irreducible constituents
of $\lambda^G$ is at most $|G:A|$ since
$\lambda^G(1)=|G:A|\lambda(1)$. The lemma now easily follows.
\end{proof}

In the next result, we analyze the average of $p'$-degrees of
irreducible characters in a special situation.

\begin{lemma}\label{lemma-key-normal-p-complement}
Let $p$ be a prime, $N$ be an abelian $p$-group, and $G$ be a split
extension of $N$. Assume that no nonprincipal irreducible character
of $N$ is fixed under $G$. Then
$$ \acd_{p'}(G)\geq \left\{\begin
{array}{ll}
3/2 & \text{ if } p=2,\\
4/3 & \text{ if } p>2.
\end {array} \right.$$
\end{lemma}

\begin{proof}
The sum of orbit sizes of the action of $G$ on nontrivial
irreducible characters of $N$ is $|N|-1$. Since $N$ is an abelian
$p$-group, there must be at least one nontrivial orbit of size
coprime to $p$. Let $\{1_N=\alpha_0,\alpha_1,\ldots,\alpha_l\}$ be a
set of representatives of $p'$-size orbits of the action of $G$ on
$\Irr(N)$. For each $0\leq i\leq l$, let $I_i$ be the inertia
subgroup of $\alpha_i$ in $G$. Then $|G:I_i|$ is not divisible by
$p$. Moreover, since no nonprincipal irreducible character of $N$ is
invariant under $G$, we have that $I_i$ is a proper subgroup of $G$
for every $1\leq i\leq l$.

Since $G$ splits over $N$, every $I_i$ also splits over $N$, and
thus $\alpha_i$ extends to a linear character, say $\beta_i$, of
$I_i$. Gallagher's theorem then implies that the mapping
$\lambda\mapsto \lambda\beta_i$ is a bijection from $\Irr(I_i/N)$ to
the set of irreducible characters of $I_i$ lying above $\alpha_i$.
Using Clifford correspondence, we then obtain a bijection
$\lambda\mapsto (\lambda\beta_i)^G$ from $\Irr(I_i/N)$ to the set of
irreducible characters of $G$ lying above $\alpha_i$. We observe
that, since $(\lambda\beta_i)^G(1)=|G:I_i|\lambda(1)$ and $p\nmid
|G:I_i|$, $(\lambda\beta_i)^G(1)$ is coprime to $p$ if and only if
$\lambda(1)$ is coprime to $p$.

From the above analysis, we see that $|\Irr_{p'}(G)|=\sum_{i=0}^l
|\Irr_{p'}(I_i/N)|$, and therefore
\[
\sum_{\chi\in\Irr_{p'}(G)}\chi(1)=\acd_{p'}(G)\sum_{i=0}^l
|\Irr_{p'}(I_i/N)|.
\]

On the other hand, since each irreducible character of $G$ lying
above $\alpha_i$ has degree at least $|G:I_i|$ and the number of
those characters of $p'$-degree is precisely equal to
$|\Irr_{p'}(I_i/N)|$, we have
\[
\sum_{\chi\in\Irr_{p'}(G)}\chi(1)\geq
\sum_{i=0}^l|G:T_i||\Irr_{p'}(I_i/N)|.
\]
We therefore deduce that
\[
\sum_{i=0}^l|G:T_i||\Irr_{p'}(I_i/N)|\leq \acd_{p'}(G)\sum_{i=0}^l
|\Irr_{p'}(I_i/N)|,
\]
which implies that
\[
\sum_{i=1}^l(|G:I_i|-\acd_{p'}(G))|\Irr_{p'}(I_i/N)|\leq
(\acd_{p'}(G)-1)|\Irr_{p'}(G/N)|
\]
since $I_0=G$. In particular, as $l\geq 1$, it follows that
\[
(\acd_{p'}(G)-1)|\Irr_{p'}(G/N)|\geq
(|G:I_1|-\acd_{p'}(G))|\Irr_{p'}(I_1/N)|.
\]
Since $|\Irr_{p'}(G/N)|\leq |G:I_1||\Irr_{p'}(I_1/N)|$ by
Lemma~\ref{lemma-compare}, we then deduce that
\[
(\acd_{p'}(G)-1)|G:I_1|\geq |G:I_1|-\acd_{p'}(G).
\]
Equivalently, we obtain
\[
\acd_{p'}(G)\geq \frac{2|G:I_1|}{|G:I_1|+1}.
\]

Recall that $|G:I_1|$ is not equal to 1 and not divisible by $p$.
Now if $p=2$ then $|G:I_1|\geq 3$ and we have $\acd_{p'}(G)\geq
3/2$. On the other hand, if $p>2$ then $|G:I_1|\geq 2$ and we have
$\acd_{p'}(G)\geq 4/3$. The proof is now complete.
\end{proof}

We are now able to prove the main Theorem~\ref{theorem-main-1},
which is restated below.

\begin{theorem}\label{theorem-main-1-again}
Let $p$ be an odd prime and $G$ a finite group. We have
\begin{enumerate}
\item[(i)] if $\acd_{2'}(G)<3/2$ then $G$ has a normal
$2$-complement, and
\item[(ii)] if $\acd_{p'}(G)<4/3$ then $G$ has a normal
$p$-complement.
\end{enumerate}
\end{theorem}

\begin{proof}
Let $b_p:=3/2$ if $p=2$ and $b_p:=4/3$ if $p>2$. Assume that
$\acd_{p'}(G)<b_p$, and we wish to show that $G$ has a normal
$p$-complement. If $G$ is abelian then the statement is obvious. So
we assume that $G$ is nonabelian. We then can choose a minimal
normal subgroup $N$ of $G$ such that $N\subseteq G'$. Since
$\acd_{p'}(G)<b_p\leq 3/2$, Theorem~\ref{theorem-main-2} implies
that $G$ is solvable, and hence $N$ is elementary abelian.

Since $N\subseteq G'$, we observe that if $\chi$ is a linear
character of $G$, then $N\subseteq \Ker(\chi)$ so that $\chi$ can be
viewed as a linear character of $G/N$. It follows that
$n_1(G/N)=n_1(G)$, which implies that
\[
\acd_{p'}(G/N)\leq \acd_{p'}(G)<b_p.
\]
Working by induction on $|G|$, we have that $G/N$ has a normal
$p$-complement, say $H/N$. If $N$ is a $p'$-group, then $H$ is a
normal $p$-complement in $G$ and we would be done. So we assume that
$N$ is an elementary abelian $p$-group. It then follows from the
Schur-Zassenhaus theorem that $H$ splits over $N$. Let us assume
that $H=NH_1$, where $H_1$ is a Hall $p'$-subgroup of $H$ (and
indeed of $G$ as well).

We now employ Frattini's argument to show that $G=N\bN_G(H_1)$. Let
$g$ be any element of $G$. Since $H\unlhd G$ and $H_1< H$, we have
$g^{-1}H_1g <H$ so that $g^{-1}H_1g$ is also a Hall $p'$-subgroup of
$H$. By Hall's theorems, $g^{-1}H_1g$ is $H$-conjugate to $H_1$. In
other words, there exists $h\in H$ such that
$g^{-1}H_1g=h^{-1}H_1h$. Thus $gh^{-1}\in \bN_G(H_1)$ so that $g\in
\bN_G(H_1)H=H\bN_G(H_1)$. Since $g$ is arbitrary in $G$, we deduce
that $G=H\bN_G(H_1)$, and therefore $G=N\bN_G(H_1)$ as $H=NH_1$.

Since $G=N\bN_G(H_1)$, if $N$ is contained in the Frattini subgroup
of $G$, we would have $G=\bN_G(H_1)$ and we are done. So we assume
that $N$ is not contained in the Frattini subgroup of $G$. Then
there exists a maximal subgroup $M$ of $G$ such that $N\nsubseteq
M$. We then have $G=NM$ and $N\cap M<N$. As $N$ is abelian, it
follows that $N\cap M$ is a normal subgroup of $G$, and hence $N\cap
M=1$ by the minimality of $N$. We conclude that $G=N\rtimes M$. In
other words, $G$ is a split extension of $N$.

If $N\subseteq \bZ(G)$ then we would have $H=N\times H_1$ and thus
$H_1\lhd G$, as desired. So we assume that $N$ is noncentral in $G$.
Thus, by the minimality of $N$, we have $[N,G]=N$. It follows that
no nonprincipal irreducible character of $N$ is invariant under $G$.

We now have all the hypotheses of
Lemma~\ref{lemma-key-normal-p-complement}, and therefore we deduce
that $\acd_{p'}(G)\geq 3/2$ if $p=2$ and $\acd_{p'}(G)\geq 4/3$ if
$p>2$. This contradiction completes the proof of the theorem.
\end{proof}


\section{$\QQ_p$-valued characters of $p'$-degree and
solvability}\label{section-Qp-solvability}

We need the following.

\begin{lemma}\label{lemma-Qp-valued-solvability}
Let $G$ be a finite group with a nonabelian minimal normal subgroup
$N$. Assume that there exists $\psi\in\Irr(N)$ that is extendible to
a $\QQ_p$-valued character of $\Stab_G(\psi)$. Then
$n_{\QQ_p,1}(G)\leq n_{\QQ_p,a}(G)|G:\Stab_G(\psi)|$, where
$a:=\psi(1)|G:\Stab_G(\psi)|$. Moreover, if $\psi$ extends to a
rational-valued character of $\Stab_G(\psi)$, then $n_{\QQ,1}(G)\leq
n_{\QQ,a}(G)|G:\Stab_G(\psi)|$.
\end{lemma}

\begin{proof}
Assume that $\psi$ extends to $\chi\in\Irr_{\QQ_p}(\Stab_G(\psi))$.
Remark that, if $\lambda$ is a linear character of $\Stab_G(\psi)/N$
with values in $\QQ_p$, then $(\lambda\chi)^G\in\Irr(G)$ has values
in $\QQ_p$ as well. Now one just repeats the arguments in the proof
of Proposition~\ref{proposition-n1-n2}(1) to obtain the first
statement of the lemma. The second statement is argued similarly.
\end{proof}

To prove Theorem~\ref{theorem-main-3}, we first prove a
$\QQ_p$-analogue of Theorem~\ref{theorem-main-2}. The next result is
an extension of \cite[Theorem~A(i)]{imn},
\cite[Theorem~C(i)]{Navarro-Tiep1}, and
\cite[Theorem~6.3]{Navarro-Tiep2}.

\begin{theorem}\label{theorem-Qp-valued-solvability}
Let $p>2$ be a prime and $G$ a finite group. If one of the following
happens
\begin{enumerate}
\item[(i)] $\acd_{\QQ,2'}(G)<3$,
\item[(ii)] $\acd_{\QQ_p,p'}(G)\leq 2$,
\item[(iii)] $\acd_{\QQ,p'}(G)\leq 2$ for $p> 3$,
\end{enumerate} then $G$ is solvable.
\end{theorem}

\begin{proof}
We use Theorem~\ref{theorem-psi-p-odd} and
Lemma~\ref{lemma-Qp-valued-solvability}, and argue as in the proof
of Theorem~\ref{theorem-odd<3-then-G-solvable} to prove (i) and
(ii).

Now we assume that $p\neq 3$ and prove (iii). By
\cite[Theorem~6.2]{Navarro-Tiep2}, the orbit $\mathcal{O}$ in
Lemma~\ref{lemma-orbit} can be chosen so that every $\theta\in
\mathcal{O}$ is extendible to a rational-valued character of
$\Stab_{\Aut(S)}(\theta)$. Therefore, the character $\psi$ produced
in Theorem~\ref{theorem-psi-p-odd} is also extendible to a
rational-valued character of $\Stab_S(\psi)$. The proof now follows
as before.
\end{proof}


\section{$\QQ_p$-valued characters of $p'$-degree and
$p$-nilpotency}\label{section-Qp-nilpotency}

We begin with an easy observation, which is recalled to our
attention by Mark L.~Lewis.

\begin{lemma}\label{lemma-Qp-valued}
Let $p$ be a prime, $N$ be an elementary abelian $p$-group, and $G$
be a split extension of $N$. Let $\theta\in\Irr(N)$ be invariant
under $G$. Then $\theta$ extends to a $\QQ_p$-valued character of
$G$.
\end{lemma}

\begin{proof}
Let $K:=\Ker(\theta)$. Since $\theta$ is $G$-invariant, $K$ is
normal in $G$. Note that $N/K$ is cyclic since it is abelian and has
a faithful irreducible character, so $\theta$ being $G$-invariant
will imply that $N/K$ is central in $G/K$. Thus, $G/K = N/K \times
HK/K$, where $H$ is a complement for $N$ in $G$. It follows that
$\theta$, viewed as a character of $N/K$, extends to $\theta \times
1_{HK/K}\in\Irr(G/K)$. Now we are done by viewing $\theta \times
1_{HK/K}$ as a character of $G$ and noting that $\theta$ has values
in $\QQ_p$.
\end{proof}

The next result is a $\QQ_p$-analogue of
Lemma~\ref{lemma-key-normal-p-complement}, but the proof is somewhat
different.

\begin{lemma}\label{lemma-key-normal-p-complement-Qp}
Let $p$ be a prime, $N$ be an elementary abelian $p$-group, and $G$
be a split extension of $N$. Assume that no nonprincipal irreducible
character of $N$ is fixed under $G$. Then
$$ \acd_{\QQ_p,p'}(G)\geq \left\{\begin
{array}{ll}
3/2 & \text{ if } p=2,\\
4/3 & \text{ if } p>2.
\end {array} \right.$$
\end{lemma}

\begin{proof}
We use the same setup as in the proof of
Lemma~\ref{lemma-key-normal-p-complement}. In particular,
$\{1_N=\alpha_0,\alpha_1,...,\alpha_l\}$ is a set of representatives
of the $p'$-size orbits of the action of $G$ on $\Irr(N)$, and $I_i$
is the inertia subgroup of $\alpha_i$ in $G$ for every $0\leq i\leq
l$.

By Lemma~\ref{lemma-Qp-valued}, each $\alpha_i$ extends to a
$\QQ_p$-valued character, say $\beta_i$, of $I_i$. Therefore, each
irreducible character of $p'$-degree of $G$ has the form $(\lambda
\beta_i)^G$ where $\lambda\in\Irr_{p'}(I_i/N)$.

Since no nonprincipal irreducible character of $N$ is fixed under
$G$, $(\lambda \beta_i)^G(1)=|G:I_i|\lambda(1)>1$ for every $1\leq
i\leq l$. Thus, every linear character of $G$ must lie above the
trivial character of $N$. We deduce that
\[
n_{\QQ_p,1}(G)=n_{\QQ_p,1}(G/N).
\]
Since $n_{\QQ_p,1}(G/N)\leq n_{\QQ_p,1}(I_1/N)|G:I_1|$, we then
obtain
\[
n_{\QQ_p,1}(G)\leq n_{\QQ_p,1}(I_1/N)|G:I_1|.
\]

Recall that $\beta_1$ has values in $\QQ_p$. Therefore if $\lambda$
is a $\QQ_p$-valued linear character of $I_1/N$, then so is
$(\lambda \beta_1)^G$, whose degree is $|G:I_1|$. We deduce that
\[
n_{\QQ_p,1}(I_1/N)\leq n_{\QQ_p,|G:I_1|}(G).
\]
Together with the above inequality, we have
\[
n_{\QQ_p,1}(G)\leq n_{\QQ_p,|G:I_1|}(G)|G:I_1|.
\]

When $p=2$ we have $|G:I_1|\geq 3$ since $|G:I_1|$ is not 1 and
coprime to $p$. It follows that \[n_{\QQ_p,1}(G)\leq
n_{\QQ_p,|G:I_1|}(G)(2|G:I_1|-3),\] and thus $\acd_{\QQ_p,2'}(G)\geq
3/2$, as claimed. On the other hand, if $p>2$ then $|G:I_1|\geq 2$,
and hence \[n_{\QQ_p,1}(G)\leq n_{\QQ_p,|G:I_1|}(G)(3|G:I_1|-4),\]
which implies that $\acd_{\QQ_p,2'}(G)\geq 4/3$, and we are done.
\end{proof}

Finally we prove the main Theorem~\ref{theorem-main-3}.

\begin{theorem}\label{theorem-main-3-again}
Let $p$ be an odd prime and $G$ a finite group. Then
\begin{enumerate}
\item[(i)] if $\acd_{\QQ,2'}(G)<3/2$ then $G$ has a normal
$2$-complement, and
\item[(ii)] if $\acd_{\QQ_p,p'}(G)<4/3$ then $G$ has a normal
$p$-complement.
\end{enumerate}
\end{theorem}

\begin{proof}
Repeat the arguments in the proof of
Theorem~\ref{theorem-main-1-again}, with the help of
Theorem~\ref{theorem-Qp-valued-solvability} and
Lemma~\ref{lemma-key-normal-p-complement-Qp}.
\end{proof}



\begin{thebibliography}{ABCD}

\bibitem[Bl]{bli}
H.\,F. Blichfeldt, {\it Finite Collineation Groups}, The University of Chicago Press, 1917.


\bibitem[Atl]{Atl1}
  J.\,H. Conway, R.\,T. Curtis, S.\,P. Norton, R.\,A. Parker, and R.\,A. Wilson,
 {\it Atlas of Finite Groups}, Clarendon Press, Oxford, 1985.

\bibitem[CHMN]{Cossey-Halasi-Maroti-Nguyen}
J.\,P. Cossey, Z. Halasi, A. Mar\'{o}ti, and H.\,N. Nguyen, On a
conjecture of Gluck, \emph{Math. Z.} \textbf{279} (2015), 1067-1080.

\bibitem[CN]{Cossey-Nguyen}
J.\,P. Cossey and H.\,N. Nguyen, Controlling composition factors of
a finite group by its character degree ratio, \emph{J. Algebra}
\textbf{403} (2014), 185-200.

\bibitem[HHN]{Halasi-Hannusch-Nguyen}
Z. Halasi, C. Hannusch, and H.\,N. Nguyen, The largest character
degrees of the symmetric and alternating groups, submitted. {\tt
http://arxiv.org/abs/1410.3055}


\bibitem[HT]{Hung-Tiep}
N.\,N. Hung and P.\,H. Tiep, Irreducible characters of even degree
and normal Sylow $2$-subgroups, submitted, 2015.

\bibitem[Is]{Isaacs1}
I.\,M. Isaacs, {\it Character theory of finite groups}, AMS Chelsea
Publishing, Providence, Rhode Island, 2006.

\bibitem[ILM]{Isaacs-Loukaki-Moreto}
I.\,M. Isaacs, M. Loukaki, and A. Moret\'{o}, The average degree of
an irreducible character of a finite group, \emph{Israel J.
Math.}~{\bf 197} (2013), 55-67.

\bibitem[IMN]{imn}
I.\,M. Isaacs, G. Malle, and G. Navarro, Real characters of $p'$-degree, {\it J. Algebra} {\bf 278} (2004), 611-620.

\bibitem[KT]{Magaard-Tongviet}
K. Magaard and H.\,P. Tong-Viet, Character degree sums in finite
nonsolvable groups, {\it J. Group Theory}~{\bf 14} (2011), 54-57.

\bibitem[LN]{Lewis-Nguyen}
M.\,L. Lewis and H.\,N. Nguyen, The character degree ratio and
composition factors of a finite group, {\it Monatsh. Math.}, to
appear. ISSN (Online) 1436-5081, ISSN (Print) 0026-9255, DOI
10.1007/s00605-015-0752-5, March 2015.


\bibitem[Lub]{Lubeck} F. L\"{u}beck, Smallest degrees of representations of exceptional groups of Lie type, {\it Comm. Algebra} \textbf{29} (2001),
2147-2169.

\bibitem[MaN]{Maroti-Nguyen}
A. Mar\'{o}ti and H.\,N. Nguyen, Character degree sums of finite
groups, \emph{Forum Math.}, to appear. ISSN (Online) 1435-5337, ISSN
(Print) 0933-7741, DOI: 10.1515/forum-2013-0066, October 2013


\bibitem[Mat]{Mattarei}
S. Mattarei, On character tables of wreath products, \emph{J.
Algebra}~\textbf{175} (1995), 157-178.

\bibitem[MoN]{Moreto-Nguyen}
A. Moret\'{o} and H.\,N.~Nguyen, On the average character degree of
finite groups, \emph{Bull. Lond. Math. Soc.} \textbf{46} (2014),
454-462.

\bibitem[Nav]{Navarro}
G. Navarro, Problems in character theory. Character theory of finite
groups, 97–125, Contemp. Math. \textbf{524}, Amer. Math. Soc.,
Providence, RI, 2010.

\bibitem[NS]{Navarro-Sanus}
G. Navarro and L. Sanus, Rationality and normal $2$-complements,
\emph{J. Algebra} \textbf{320} (2008), 2451-2454.

\bibitem[NST]{Navarro-Sanus-Tiep}
G. Navarro, L. Sanus, and P.\,H. Tiep, Real characters and degrees,
\emph{Israel J. Math.} \textbf{171} (2009), 157-173.

\bibitem[NT1]{Navarro-Tiep1}
 G. Navarro and P.\,H. Tiep, Characters of $p'$-degree with cyclotomic field of values, \emph{Proc. Amer. Math. Soc.} \textbf{134} (2006), 2833-2837.

\bibitem[NT2]{Navarro-Tiep2}
 G. Navarro and P.\,H. Tiep, Degrees of rational characters of finite groups, \emph{Adv. Math.} \textbf{224} (2010), 1121-1142.

\bibitem[Ng]{Nguyen1}
  H.\,N. Nguyen, Low-dimensional complex characters of the
  symplectic and orthogonal groups, {\it Comm. Algebra} {\bf 38} (2010),
  1157-1197.

\bibitem[Qia]{Qian}
G. Qian, On the average character degree and the average class size
in finite groups, \emph{J. Algebra} \textbf{423} (2015), 1191-1212.

\bibitem[Ras]{Rasala}
R. Rasala, On the minimal degrees of characters of $\Sy_n$, {\it J.
Algebra}~{\bf 45} (1977), 132--181.

\bibitem[Th]{Thompson}
J.\,G.~Thompson, Normal $p$-complements and irreducible characters,
\emph{J. Algebra} \textbf{14} (1970), 129-134.

\bibitem[TZ]{Tiep-Zalesskii}
  P.\,H. Tiep and A.\,E. Zalesskii, Minimal characters of the finite classical
  groups, {\it Comm. Algebra}~{\bf 24} (1996), 2093-2167.



\end{thebibliography}
\end{document}